%% file: supercharacters.tex
\documentclass[fontsize=12pt,paper=a4,toc=bib]{scrartcl}
\pdfoutput=1

\usepackage[T1]{fontenc}

\usepackage[style=numams, 
            sorting=nyt,
            isbn=false,  
            backref=true, 
            maxbibnames=4,
            backend=biber]{biblatex}
\usepackage{amsmath,amssymb, amsthm}
\usepackage{mathtools}
\usepackage{enumitem}

\usepackage{tikz-cd}

\usepackage{ctable}

\usepackage[colorlinks, linkcolor=blue, citecolor=blue, pdfa]{hyperref}

\usepackage[capitalize,nameinlink]{cleveref}

\input{layout_scrartcl}

\newtheorem{introthm}{Theorem}

\swapnumbers
\newtheorem{thm}{Theorem}[section]
\newtheorem{cor}[thm]{Corollary}
\newtheorem{lemma}[thm]{Lemma}
\newtheorem{question}[thm]{Question}
\newtheorem{problem}[thm]{Problem}
\newtheorem{conjecture}[thm]{Conjecture}

\theoremstyle{definition}
\newtheorem{defi}[thm]{Definition}
\newtheorem{example}[thm]{Example}

\newtheorem{remark}[thm]{Remark}
\newtheorem{algo}[thm]{Algorithm}
\setlist{nosep}

\newlist{enums}{enumerate}{2}  
\setlist[enums,1]{label=\textup{(\alph*)}}

\crefname{enumsi}{Part}{Parts}
\Crefname{enumsi}{Part}{Parts}

\renewcommand{\phi}{\varphi}
\renewcommand{\theta}{\vartheta}

\renewcommand{\geq}{\geqslant}
\renewcommand{\leq}{\leqslant}

\newcommand{\defemph}[1]{\textbf{#1}} 

\newcommand{\nats}{\mathbb{N}}
\newcommand{\ints}{\mathbb{Z}}
\newcommand{\compl}{\mathbb{C}}
\newcommand{\iso}{\cong}    

\newcommand{\dcup}{\mathbin{\mathaccent\cdot{\cup}}}   

\DeclarePairedDelimiter{\abs}{\lvert}{\rvert}

\DeclareMathOperator{\Z}{\mathbf{Z}}         

\DeclareMathOperator{\Irr}{Irr}
\DeclareMathOperator{\Cl}{Cl}
\DeclareMathOperator{\Map}{Maps}
\DeclareMathOperator{\cf}{cf}

\DeclareMathOperator{\Clpt}{ClPt}
\DeclareMathOperator{\Irpt}{IrPt}

\DeclareMathOperator{\PSL}{PSL}
\DeclareMathOperator{\Sz}{Sz}  
\DeclareMathOperator{\UT}{UT}
\DeclareMathOperator{\Aut}{Aut}

\newcommand{\noids}[1]{#1^{\#}}   

\newcommand{\pttf}[1]{\mathcal{#1}}
\newcommand{\sct}[2]{(\pttf{#1},\pttf{#2})}
\newcommand{\px}{\pttf{X}}
\newcommand{\pk}{\pttf{K}}

\hypersetup{pdfauthor={Frieder Ladisch},
   pdftitle={Finding Supercharacter Theories on Character Tables},
   pdfkeywords={Characters, Finite groups, Supercharacters, Schur rings}
           }

\begin{document}
\author{Frieder Ladisch}
\title[Supercharacter Theories]{Finding Supercharacter Theories \\
       on Character Tables}
\address{Universität Rostock \\
         Institut für Mathematik\\         
         18051 Rostock\\
         Germany}
\email{frieder.ladisch@uni-rostock.de}


\maketitle

\begin{abstract}
    We describe an easy way how to find supercharacter theories
    for a finite group,
    if its character table is known.
    Namely, 
    we show how an arbitrary partition of the conjugacy classes 
    or of the irreducible characters can be refined to the 
    coarsest partition that belongs to a supercharacter theory.
    Our constructions emphasize the duality between 
    superclasses and supercharacters.
    An algorithm is presented to find all supercharacter theories
    on a given character table.
    The algorithm is used to compute
    the number of supercharacter theories for some 
    nonabelian simple groups with up to 26 conjugacy classes.
\end{abstract}


%


\section{Introduction}

Supercharacter theories of finite groups were introduced by 
Diaconis and Isaacs~\cite{DiaconisIsaacs08}
as approximations of a group's 
ordinary character theory.
Let $G$ be a finite group
and write $\Irr G$ to denote the set of irreducible
complex characters of $G$.
For any subset $X\subseteq \Irr G$, let
$\sigma_X$ be the character
\[ \sigma_X = \sum_{\chi\in X} \chi(1)\chi. 
\]
Suppose that $\px$ is a partition of $\Irr G$
and $\pk$ is a partition of $G$.
The pair $\sct{X}{K}$
is called
a \defemph{supercharacter theory} of $G$,
\label{def:sct}
if the following conditions hold:
\begin{enums}
\item $\abs{\px}=\abs{\pk}$,
\item the characters $\sigma_X$ ($X\in \px$)
      are constant on the members of $\pk$.      
\end{enums}
The members of $\pk$ are called the 
\defemph{superclasses} of the theory $(\px,\pk)$,
and the characters $\sigma_X$ 
(or certain rational multiples) the 
\defemph{supercharacters}.

We say that $\pk$ is
\defemph{compatible} with $\px$,
if $\sigma_X$ is constant on the members of $\pk$
for every $X \in \px$.
Clearly, for every collection $\px$ 
of subsets of $\Irr G$,
there is a unique coarsest partition  $\pk $ 
of $G$ that is compatible with $\px$:
Namely, 
$\pk$ is the partition of $G$ whose members
are the equivalence classes under the relation 
on $G$ defined by $g \sim h$ if and only if
$\sigma_X(g)=\sigma_X(h)$ for all $X\in \px$.
We write 
\[ \pk = \Clpt(\px) 
\]
for this partition (as in \emph{class partition}).

We always have $\abs{\px} \leq \abs{\Clpt(\px)}$,
and the partition $\px$ belongs to a supercharacter theory
if and only if equality holds. 
Thus in a supercharacter theory, $\px$ determines 
$\pk = \Clpt(\px)$
as the coarsest partition of $G$
that is compatible with $\px$.
This description appears in the paper of Diaconis and Isaacs
\cite[Theorem~2.2~(c)]{DiaconisIsaacs08}.

In this note (in \cref{def:IrPt} below), 
we show how each collection~$\pk$ 
of normal subsets of $G$ determines a partition~$\px$ of $\Irr G$,
thereby defining a map
\[  \pk \mapsto \Irpt(\pk) = \px 
\]
which is in some sense dual to the map $\Clpt$ described above.
With both these maps in hand, 
we can easily characterize partitions belonging to 
supercharacter theories, 
and we also get an easy way to refine an arbitrary partition
of $\Irr G$ or a $G$-invariant partition of $G$
to a partition belonging to a supercharacter
theory.
(A partition $\pk$ of $G$ is called
\defemph{$G$-invariant} if all its members are normal subsets of $G$.)
\begin{introthm}
\label{t:main}
  Let $\pk$ a $G$-invariant partition of $G$, and 
  $\px$ a partition of $\Irr G$.
  \begin{enums}
  \item \label{i:pktopx}
        $\abs{\pk} \leq \abs{ \Irpt(\pk)}$,
        and equality holds if and only if
        $\big( \Irpt(\pk), \pk \big)$ 
        is a supercharacter theory.
  \item \label{i:diaisa22c}
        $\abs{\px} \leq \abs{ \Clpt(\px)}$, 
        and equality holds if and only if
        $ \big( \px, \Clpt(\px) \big)$
        is a supercharacter theory.
  \item \label{i:pkrefine}
        The partition 
        $\Clpt\big( \Irpt( \pk) \big)$ refines $\pk$,
        and these partitions are equal if and only if
        $\big( \Irpt(\pk), \pk \big)$ 
        is a supercharacter theory.
  \item \label{i:pxrefine} The partition 
        $\Irpt\big( \Clpt( \px)\big)$ refines $\px$,
        and these partitions are equal if and only if
        $ \big( \px, \Clpt(\px) \big)$
        is a supercharacter theory.
  \end{enums}
\end{introthm}

\Cref{i:diaisa22c} is just the result of Diaconis and Isaacs
mentioned before, and is included just to emphasize
duality with~\cref{i:pktopx}.
While it is also known that $\pk$ determines $\px$ when
$(\px,\pk)$ is a supercharacter theory,
our description of $\px$ as $\Irpt(\pk)$ is different from the one
contained in the literature.
These results are proved in \cref{sec:ClPtIrPt}.

The maps $\Clpt$ and $\Irpt$ yield an efficient way to refine a given 
partition of $G$ or of $\Irr G$ to a partition
belonging to a supercharacter theory:
For example,
suppose we are given a $G$-invariant partition~$\pk$ of $G$.
We apply the maps $\Irpt$ and $\Clpt$ repeatedly in turns.
By \cref{i:pktopx,i:diaisa22c} of \cref{t:main},
we have
\[ \abs{\pk} \leq \abs{ \Irpt(\pk)}
   \leq \abs{\Clpt(\Irpt(\pk))}
   \leq \abs{ \Irpt(\Clpt(\Irpt(\pk)))}
   \leq \dotsb \,. 
\]
When two consecutive partitions contain the same number of sets,
then these
two partitions form a supercharacter theory.
Moreover, by \cref{i:pkrefine},
the corresponding superclass partition is the coarsest partition
that refines $\pk$ and belongs to a supercharacter theory.
Computational evidence suggests that most often, 
a small number of steps suffices to reach a supercharacter theory.
Also, both maps $\Clpt$ and $\Irpt$ can be computed easily and 
effectively from the character table.

Similarly, we can start with a partition $\px$ of $\Irr G$
and refine it to a partition of $\Irr G$ 
that belongs to a supercharacter theory.

So when we know the character table of a group $G$,
we can easily decide whether a given nonempty, normal subset $S$
of $G$ is a superclass in some supercharacter theory on $G$:
We apply $\Irpt$ and $\Clpt$ to the partition
$\{S, G\setminus S\}$ of $G$,
until either $S$ is not a member of the resulting partition of $G$,
or we have found a supercharacter theory $(\px,\pk)$ with $S\in \pk$.
In the first case, $S$ can not be a superclass. 
In the second case, we have also found 
the coarsest supercharacter theory in which $S$ is a superclass.

Although the original motivation for developing supercharacter theories
were cases where the full character table is difficult 
to compute, 
there has recently been some interest in understanding 
all possible supercharacter theories on given groups
and character tables.
Only for a few families of finite groups are all possible supercharacter
theories known: 
Leung and Man~\cite{LeungMan96,LeungMan98} classified supercharacter
theories of finite cyclic groups in the language of Schur rings
(cf. Hendrickson's paper~\cite{Hendrickson12}).
Wynn~\cite{Wynn17} and Lamar~\cite{Lamar18} both classified 
supercharacter theories of dihedral groups,
and Lewis and Wynn~\cite{LewisWynn20,Wynn17} considered
supercharacter theories of Camina pairs, and in particular classified 
them for Frobenius groups of order $pq$.
Recently, Burkett and Lewis~\cite{BurkettLewis20pre} began a classification of 
supercharacter theories of $C_p\times C_p$.

Our results suggest an algorithm for determining all
supercharacter theories on a given character table
(\cref{algo:allscts}).
In a first step, this algorithm runs through half of the 
nonempty, normal subsets $S$ of $G\setminus\{1\}$.
As described above, we can at the same time
decide whether such an $S$ can be a superclass,
and compute the coarsest supercharacter theory 
$(\px,\pk)$ with $S\in \pk$, if there is such a theory at all.
In a second step, the algorithm forms meets of the supercharacter 
theories found in the first step, and adds the trivial 
supercharacter theory with class partition
$\pk = \{ \{1\}, G\setminus\{1\}\}$.

The algorithm runs through
$2^{k(G)-2}-1$ subsets of $G$,
and thus the algorithm is applicable only for groups
with few conjugacy classes.
But our algorithm is more efficient than the one suggested by
Hendrickson~\cite{Hendrickson08}, or the modified version by
Burkett, Lamar, Lewis and Wynn~\cite{BurkettLLW17,Lamar18}.
(See the remarks at the end of \cref{sec:find_all_scts} below.)
A further, small improvement is possible by using
automorphisms of the character table,
which we discuss briefly in \cref{sec:tableauts}.

Using the character tables from the character table library
of GAP~\cite{GAP},
we have computed the number of supercharacter theories
for some groups with up to 26 conjugacy classes.
The results are summarized at in \cref{sec:nrscts}.
For a group with about 24 conjugacy classes, 
the algorithm needs a few minutes on a standard desktop computer.

For example, consider the second Janko $J_2$.
This group has 21 conjugacy classes.
It is unfeasible (with current technology)
to run through all the 51\,724\,158\,235\,372 $G$-invariant partitions
of $G$ containing $\{1\}$ as a block,
and check for each partition whether it belongs to a supercharacter 
theory.
But our algorithm has only to run through 
$2^{19}-1 = 524287 $ $G$-invariant subsets of $G\setminus \{1\}$.
(The second step of the algorithm is trivial for $J_2$ since only two
supercharacter theories are found in the first step.)
In about 100 seconds, 
the algorithm finds that the second Janko group $J_2$ has exactly three 
 supercharacter theories.
This confirms a conjecture by A.~R.~Ashrafi and
F.~Koorepazan-Moftakhar
\cite[Conjecture~2.8]{AshrafiMoftakhar16pre}.
According to the published version of this
paper~\cite[remarks before Lemma~3.9]{AshrafiMoftakhar18},
this fact has also been established by 
N. Thiem and J. P. Lamar in unpublished work,
but their algorithm needs about 3 hours.


\section{Partitions and algebras of maps}
\label{sec:parttalg}

We begin with a very general,
but well known and elementary result.
Let $S$ be a finite set and let
$\pk$ and $\pttf{L}$ be partitions of $S$.
Recall that $\pk$ is said to be
\defemph{finer} than $\pttf{L}$, 
written $\pk \preceq \pttf{L}$,
when every block of $\pk$
(that is, every set $K\in \pk$)
is contained in a block of $\pttf{L}$.
One also says that $\pttf{L}$ is 
\defemph{coarser} than $\pk$ in this case.
The set of all partitions of  $S$ 
forms a partial ordered set under the relation of
refinement, and in fact a  lattice.

Now let $F$ be a field and $S$ a finite set.
The set $F^S=\Map(S,F)$ of all functions 
$f\colon S\to F$ forms a commutative $F$-algebra with $1$
with respect
to pointwise addition and multiplication.
Let $\pk$ be a partition of $S$.
The set $\Map_{\pk}(S,F)$ of all functions 
which are constant on the blocks of $\pk$
is a unital subalgebra of $\Map(S,F)$.
Conversely, every subalgebra of $\Map(S,F)$
containing the all-$1$-function has this form.
This is the content of the following 
(well-known) lemma:

\begin{lemma}\label{l:partt_subalgs}
   The map $\pk\mapsto \Map_{\pk}(S,F)$
     defines a bijection between partitions  $\pk$ of $S$
     and unital $F$-subalgebras of\/ $\Map(S,F)$.
     The inverse sends a subalgebra $A$ to the partition 
     corresponding to the equivalence relation on $S$ 
     defined by $s\sim t $ if and only if
     $f(s)=f(t)$ for all $f\in A$.  
     
   The bijections are order reversing with respect to
     refinement and inclusion, that is,
     $\pk \preceq \pttf{L}$ if and only if
     $\Map_{\pk}(S,F) \supseteq \Map_{\pttf{L}}(S,F)$.
\end{lemma}
\begin{proof}
  It is clear that
  $\pk \mapsto \Map_{\pk}(S,F) \mapsto \pk$,
  because  for $K\in \pk$,
  the characteristic function $\delta_K$ 
  with $\delta_K(s)=1$ if $s\in K$ and $\delta_K(s)=0$ else
  is contained in $\Map_{\pk}(S,F)$.

  For the sake of completeness, we also prove 
  the converse, although all this is well known.
  Let $A\subseteq \Map(S,F)$ be a $F$-subalgebra
  containing $1$ (the all-$1$-function),
  and let $\pk$ be the partition
  whose members are equivalence classes
  under the relation
  on $S$ defined by $s \sim t$ if and only if
  $f(s)= f(t)$ for all $f \in A$.
  Obviously, $A\subseteq \Map_{\pk}(S,F)$.
  
  Let $K\in \pk$ and fix $s\in K$.
  For every $t\in S\setminus K$, there is a 
  function $f\in A$ such that $f(s)\neq f(t)$.
  Then the function 
  $g_t = \big(f(s)-f(t)\big)^{-1} \big( f - f(t)\cdot 1 \big)$ 
  is also in $A$ and we have
  $g_t(s)=1$  and $g_t(t) = 0$.
  Since $g_t$ is constant on $K$, we have $g_t(x)=1$ for 
  all $x\in K$.
  Multiplying all $g_t$ for $t\in S\setminus K$,
  we get the characteristic function $\delta_K$ of $K$.
  It follows that $A$ contains the characteristic functions
  $\delta_K$ for all $K\in \pk$, and so 
  $A=\Map_{\pk}(S,F)$.
  
  The last statement of the lemma
  is easy to verify.
\end{proof}

\begin{cor}\label{c:gensubalg}
  Let $B \subseteq \Map(S,F)$ 
  be a set of maps.
  The unital subalgebra of $\Map(S,F)$ generated by
  $B$ is $\Map_{\pk}(S,F)$,
  where $\pk$ is the partition corresponding
  to the equivalence relation on $S$ defined by
  $s\sim t$ if and only if $b(s)=b(t)$
  for all $b\in B$.  
\end{cor}


\section{Class and character partitions}
\label{sec:ClPtIrPt}

First we apply the results of the last section to the algebra
of maps $\Map(G,\compl)$ from a finite group~$G$ into 
the field of complex numbers~$\compl$.
Actually everything takes place in the subalgebra
of class functions $\cf(G)$.
These are the functions $G\to \compl$ that are constant on
conjugacy classes.

For $X\subseteq \Irr G$, set
\[ \sigma_X = \sum_{\chi\in X} \chi(1)\chi \]
as in the introduction.

\begin{defi}
For a set $\px$ consisting of subsets of $\Irr G$, 
let
$\Clpt(\px)$ 
be the partition of $G$ whose members are the equivalence classes under
the relation~$\sim_{\px}$ on $G$ defined by 
$g \sim_{\px} h$ if and only if
$\sigma_X(g)=\sigma_X(h)$ for all $X\in \px$.
\end{defi}

Recall from the introduction that we call a partition 
$\pk$ of $G$ 
\defemph{compatible} with $\px$,
if $\sigma_X$ is constant on the members of $\pk$
for every $X \in \px$.

\begin{lemma}\label{l:ax_elprops}
  Let $\px$ be a collection of subsets of $\Irr G$.
  \begin{enums}
  \item \label{i:ax_coarsest}
         $\Clpt(\px)$ is the unique coarsest partition of
  $G$ that is compatible with $\px$.   
  \item \label{i:ax_subalg}
        $\Map_{\Clpt(\px)}(G,\compl)$ 
        is the unital subalgebra of $\cf(G)$
        generated by the characters $\sigma_X$ for $X\in \px$.
  \end{enums}
  If $\px$ is a partition of $\Irr G$, then
  the following hold:
  \begin{enums}[resume]
  \item \label{i:ax_grow}
        $\abs{\px} \leq \abs{\Clpt(\px)}$.
  \item \label{i:ax_one} $\{1\} \in \Clpt(\px)$.
  \item \label{i:ax_refine} When $\pttf{Y}$ is another partition of
        $\Irr G$ with $\px \preceq \pttf{Y}$,
        then $\Clpt(\px) \preceq \Clpt(\pttf{Y})$.
  \end{enums}
\end{lemma}

As explained in the introduction,
supercharacter theories are characterized by 
equality in~\ref{i:ax_grow}. 
The results of \cref{l:ax_elprops} are well known 
and recorded here 
for convenient reference, but also for motivation of the dual
results to follow.

\begin{proof}[Proof of \cref{l:ax_elprops}] 
  \Cref{i:ax_coarsest} just rephrases the definition. 
  \Cref{i:ax_subalg}
  is immediate from \cref{c:gensubalg}.
  If $\px$ is a partition,
  then the characters $\sigma_X$ are linearly independent.
  Thus 
  $ \abs{\px} 
    \leq \dim \Map_{\Clpt(\px)}(G,\compl) 
    = \abs{\Clpt(\px)}$,
  which is~\ref{i:ax_grow}.
  \Cref{i:ax_one} follows since
  the regular character
  $\rho_G = \sum_{X\in \px}\sigma_X$ is in the span of the 
  characters $\sigma_X$.
  Finally, when $\px \preceq \pttf{Y}$, 
  then every $Y\in \pttf{Y}$
  is a union of blocks of $\px$,
  and $\sigma_Y $ is the sum of the corresponding
  $\sigma_X$'s.
  Then it is clear from the definition that
  $g \sim_{\px} h$ implies $g\sim_{\pttf{Y}} h$, as claimed.
\end{proof}

Next we want to describe how a 
$G$-invariant partition of $G$ determines a partition
of $\Irr G$.
First, we need to introduce some notation.
Recall that any 
  $\chi\in \Irr G$ 
  defines a central primitive idempotent 
  $e_{\chi}$
  of the group algebra $ \compl G$,
  namely
  \[ e_{\chi}= \frac{\chi(1)}{\abs{G}}\sum_{g\in G} \chi(g)g^{-1}.
  \]
These idempotents yield a decomposition of the center of the
group algebra: Namely, we have 
  \[\Z(\compl G)= \bigoplus_{\chi\in \Irr G}
      \Z(\compl G)e_{\chi}
      =\bigoplus_{\chi\in \Irr G} \compl e_{\chi}.
  \]  
Moreover, recall that 
  \[ \omega_{\chi}\colon \Z(\compl G) \to \compl
     , \quad 
     \omega_{\chi}(z)= \frac{\chi(z)}{\chi(1)}
     ,
  \]
  is the  \defemph{central character} 
  associated to $\chi\in \Irr G$.
  It is defined by the property
  $R(z) = \omega_{\chi}(z)I$ for $z\in \Z(\compl G)$,
  where $R$ is a representation affording $\chi$.
The central character $\omega_{\chi}$ also describes
the projection to the component $\compl e_{\chi}$
of $\Z(\compl G)$.
In other words, we have
\[ z = \sum_{\chi\in \Irr G} \omega_{\chi}(z) e_{\chi}
   \quad \text{for all } z\in \Z(\compl G).
\]
This yields the following lemma:

\begin{lemma}\label{l:zcg_irrg}
  For $z\in \Z(\compl G)$, define
  \[ \alpha_z\colon \Irr G \to \compl
     ,\quad \alpha_z(\chi) = \omega_{\chi}(z) 
     = \frac{\chi(z)}{\chi(1)}.
  \]
  The map $z\mapsto \alpha_z$
  defines an algebra isomorphism
  \[ \Z(\compl G) \iso  \Map(\Irr G, \compl).
  \]
  The inverse sends a map $f\colon \Irr G\to \compl$
  to the element $\sum_{\chi} f(\chi)e_{\chi}$.
\end{lemma}

By \cref{l:partt_subalgs},
there is a natural correspondence between
the unital subalgebras of 
$\Z(\compl G)\iso \Map(\Irr G, \compl)$
and partitions of $\Irr G$.
For $X\subseteq \Irr G$, we write
\[ e_X = \sum_{\chi \in X} e_{\chi}
       = \frac{1}{\abs{G}}\sum_{g\in G} \sigma_X(g)g^{-1}.
\]
Then the subalgebra $A_{\px}$
corresponding to a partition $\px$ of $\Irr G$
is given by
$A_{\px}=\bigoplus_{X\in \px} \compl e_X$ of $\Z(\compl G)$.

We also use the following notation:
For any subset $K$ of $G$, we write 
\[ \hat{K} := \sum_{g\in K} g \in \compl G.
\]

\begin{defi}\label{def:IrPt}
   Let $\pk$ be a collection
   of subsets of $G$, such that each member of $\pk$
   is a union of conjugacy classes of $G$.   
   Then define $\Irpt(\pk)$ to be the partition of 
   $\Irr G$ whose members are the equivalence classes
   of the relation
   on $\Irr G$ defined by $\chi \sim_{\pk} \psi $ if and only if
   $\chi(\hat{K})/ \chi(1) =\psi(\hat{K})/\psi(1)$
   for all $K\in \pk $.
\end{defi}

We have the following result, which is completely dual to
\cref{l:ax_elprops}.

\begin{lemma}\label{l:bk_elprops}
  Let $\pk$ be a collection of $G$-invariant subsets of $G$.
  \begin{enums}
  \item \label{i:bk_coarsest}
         $\Irpt(\pk)$ is the unique coarsest partition of\/
         $\Irr G$ such that for every $K\in \pk$,
         the map
         $\alpha_{\hat{K}}$ is constant on the members
         of\/ $\Irpt(\pk)$.   
  \item \label{i:bk_subalg}
        $\sum_{X\in \Irpt(\pk)} \compl e_X$ is the subalgebra of 
        $\Z(\compl G)$
        generated by the block sums $\hat{K}$ for $K\in \pk$.
  \end{enums}
  If $\pk$ is a $G$-invariant partition of $G$, then also:
  \begin{enums}[resume]
  \item \label{i:bk_grow}
        $\abs{\pk} \leq \abs{\Irpt(\pk)}$.
  \item \label{i:bk_one} $\{1_G\} \in \Irpt(\pk)$.
  \item \label{i:bk_refine} When $\pttf{L}$ is another
        $G$-invariant partition of
        $G$ with $\pk \preceq \pttf{L}$,
        then $\Irpt(\pk) \preceq \Irpt(\pttf{L})$.
  \end{enums}
\end{lemma}
\begin{proof}
    Let $K\in \pk$.
    Since $K$ is a union of conjugacy classes of $G$,
    we have $\hat{K}\in \Z(\compl G)$.
    Thus the class sums $\hat{K}$ for $K\in \pk$ generate
    a subalgebra of $\Z(\compl G) \iso \compl^{\Irr G}$,
    and this subalgebra determines a partition
    $\px$ of $\Irr G$.
    By \cref{l:partt_subalgs} and \cref{l:zcg_irrg},
    the members of $\px$ are the equivalence classes of the relation
    on $\Irr G$ defined by 
    $\chi \sim_{\pk} \psi $ if and only if
    $\alpha_{\hat{K}}(\chi) = \alpha_{\hat{K}}(\psi)$
    for all $K\in \pk$.
    By the definition of $\alpha$ in \cref{l:zcg_irrg}, 
    we have $\chi\sim_{\pk} \psi$ if and only if
    $\chi(\hat{K})/ \chi(1) =\psi(\hat{K})/\psi(1)$
    for all $K\in \pk $.
    Thus $\px=\Irpt(\pk)$,
    and \ref{i:bk_coarsest} and \ref{i:bk_subalg} follow.

    When $\pk$ is a partition of $G$,
    then the sums $\hat{K}$ are linearly independent.
    Thus \ref{i:bk_grow} follows from~\ref{i:bk_subalg}.
    To see~\ref{i:bk_one}, we use that
    for $z=\sum_{g\in G} g$,
    we have $\omega_{\chi}(z)\neq 0$ if and only if
    $\chi=1_G$.
    \Cref{i:bk_refine} is easy.
\end{proof}

Notice that \cref{i:bk_subalg} yields another characterization
of the partition~$\Irpt(\pk)$.
This is the characterization given by Diaconis and
Isaacs~\cite{DiaconisIsaacs08} in the case where
$\pk$ belongs to a supercharacter theory.

Partitions belonging to a supercharacter theory
are characterized by equality in \cref{i:ax_grow}
of \cref{l:ax_elprops} or~\cref{l:bk_elprops},
respectively.

\begin{lemma}\label{l:ab_ba_char}\hfill
  \begin{enums}
  \item \label{i:abk_suba}
        Let $\pk$ be a collection of 
        $G$-invariant subsets of $G$.
        Then $\Clpt(\Irpt(\pk))$ is the coarsest partition
        $\pttf{L}$ such that the linear span
        of the sums $\hat{L}$, where $L\in \pttf{L}$,
        contains the unital subalgebra of $\Z(\compl G)$
        generated by the sums $\hat{K}$, where $K\in \pk$.
  \item \label{i:bax_suba}
        Let $\px$ be a collection of subsets of $\Irr G$.
        Then $\Irpt(\Clpt(\px))$ is the coarsest partition
        $\pttf{Y}$ such that the linear span
        of the characters $\sigma_Y$, where $Y\in \pttf{Y}$,
        contains the unital subalgebra of $\cf(G)$
        generated by the characters $\sigma_X$,
        where $X\in \px$.
  \end{enums}
\end{lemma}

\begin{proof}
  We begin with \ref{i:abk_k}.
  Let $\px:= \Irpt(\pk)$.
  By \cref{l:bk_elprops}~\ref{i:bk_subalg},
  the idempotents $e_X$, where $X\in \px$,
  form a basis of the subalgebra generated
  by the sums $\hat{K}$, $K\in \pk$.
  But we have
  \[ e_X = \sum_{\chi\in X} e_{\chi}
         = \frac{1}{\abs{G}} \sum_{g\in G} 
              \overline{\sigma_X(g)} g.
  \]
  By definition, $\Clpt(\px)=\Clpt(\Irpt(\pk))$ is the 
  coarsest partition~$\pttf{L}$
  of $G$ such that $\sum_{L\in \pttf{L}} \compl \hat{L}$
  contains the idempotents $e_X$, as claimed.
  
  The proof of~\ref{i:bax_x} is similar.
  First, let $K\subseteq G$ be a union of conjugacy classes,
  and let $\delta_K\colon G\to \compl$ be its characteristic function.
  By the orthogonality relations,
  we have
  \[ \delta_K = \frac{1}{\abs{G}} 
                \sum_{\chi\in \Irr G} \overline{\chi(\hat{K})} \chi 
        = \frac{1}{\abs{G}} 
          \sum_{\chi \in \Irr G}  
                 \overline{ \omega_{\chi} ( \hat{K} ) }
                 \chi(1)
                 \chi.
  \]
  Now let $\pk := \Clpt(\px)$.  
  By \cref{l:ax_elprops}~\ref{i:ax_subalg}, 
  the functions $\delta_K$ where $K\in \pk$,
  form a basis of the subalgebra generated
  by the characters $\sigma_X$, where $X\in \px$.
  By the  above formula for $\delta_K$,
  the linear span of the characters $\sigma_Y$,
  where $Y$ runs through a partition $\pttf{Y}$  of $\Irr G$,
  contains all $\delta_K$ for $K\in \pk$,
  if and only if $\pttf{Y}\preceq \Irpt(\pk)=\Irpt(\Clpt(\px))$.
  This is the claim.
\end{proof}

\begin{cor} \label{c:ab_ba_ref}\hfill
  \begin{enums}
  \item \label{i:abk_k}
        Let $\pk$ be a $G$-invariant partition of $G$.
        Then  $\Clpt(\Irpt(\pk))\preceq \pk$.
  \item \label{i:bax_x}
        Let $\px$ be a partition of $\Irr G$.
        Then
        $\Irpt(\Clpt(\px)) \preceq \px$.
  \end{enums}
\end{cor}

\begin{example}
   Let $\pttf{N}$ be a collection of normal subgroups of 
   the finite group~$G$.
   Then we can apply \cref{l:ab_ba_char}~\ref{i:abk_suba}
   to $\pttf{N}$.
   We want to describe $\pttf{L}:= \Clpt(\Irpt(\pttf{N}))$.   
   
   Let $A$ be the unital subalgebra of $\Z(\compl G)$ 
   generated by the sums $\hat{N}$ for $N\in \pttf{N}$.
   For $N$, $M\in \pttf{N}$, we have
   $\hat{N}\hat{M} = \abs{N\cap M} \widehat{NM} \in A$.
   Let $\pttf{S} = \pttf{S}(\pttf{N})$ be the set of normal subgroups
   of the form $N_1 \dotsm N_r $ with $N_i\in \pttf{N} $,
   including the trivial subgroup as the empty product.
   Then for each $S\in \pttf{S}$ there is some positive
   integer $n_S$ such that $n_S\hat{S} \in A$.
   On the other hand, the sums $\hat{S}$ with $S\in \pttf{S}$
   obviously span a subalgebra of $\Z(\compl G)$.
   
   By \cref{l:ab_ba_char}~\ref{i:abk_suba},
   $\pttf{L}$ is the coarsest partition of $G$ such that 
   every $S\in \pttf{S}$ is a union of elements of $\pttf{L}$.
   Describing $\pttf{L}$ is now a matter of elementary set theory:   
   Let $\pttf{N}_1$ be the set of intersections
   of finitely many elements of $\pttf{S}$, including $G$.
   This is then a set of normal subgroups of $G$ which is closed
   under taking intersections.
   (In general, $\pttf{N}_1$ may not be closed under taking
   products.)
   The partition~$\pttf{L}$ is 
   the coarsest partition of $G$ such that 
   every $N\in \pttf{N}_1$ is a union of elements of $\pttf{L}$.
   Since $\pttf{N}_1$ is closed under intersections,
   every $g\in G$ is contained in a unique minimal
   $N\in \pttf{N}_1$.
   Let us write $g^{\pttf{N}_1}$ for this normal subgroup.
   The (nonempty) fibers of the map 
   $G\ni g \mapsto g^{\pttf{N}_1} \in \pttf{N}_1$
   are then the blocks of the partition $\pttf{L}$.
   That is, every block $L\in \pttf{L}$ has the form
   \[ L=L_N = \{g\in G \mid g^{\pttf{N}_1} = N\}
           = N \setminus 
             \bigcup_{ 
                \substack{ 
                   M\in\pttf{N}_1 \\
                   M < N
                }
             } M
       \quad \text{for some } N\in \pttf{N}_1.   
   \]
   It follows that the sums
   $\hat{N}$ with $N\in \pttf{N}_1$ and the sums
   $\hat{L}$ with $L\in \pttf{L}$ 
   span the same $\ints$-submodule of
   $\Z(\ints G)$.
   
   When $\pttf{N}_1$ is closed under products,
   then the linear span of the sums $\hat{N}$ with $N\in \pttf{N}_1$
   is actually a subalgebra.
   In general, this is not the case, but we can repeat the process
   until we have found a collection of normal subgroups that is closed
   under intersections and products.
   The corresponding partition of $G$ then belongs to a supercharacter
   theory.
   These supercharacter theories 
      were described by F.~Aliniaeifard~\cite{Aliniaeifard17}.
\end{example}

\begin{thm}\label{t:x_sct}
  Let $\px$ be a partition of $\Irr G$.
  Then the following are equivalent:
  \begin{enums}
  \item \label{i:x_equ} $\abs{\px} = \abs{\Clpt(\px)}$.
  \item \label{i:x_subalg}
         The linear span in the space of class functions
        of the characters $\sigma_X$ for $X\in \px$
        is a unital subalgebra of the algebra
        of all class functions.
  \item \label{i:x_bax} $\px= \Irpt(\Clpt(\px))$.
  \end{enums}
\end{thm}
\begin{proof}
  The equivalence of \ref{i:x_equ} and \ref{i:x_subalg}
  follows from 
  \ref{i:ax_subalg} and \ref{i:ax_grow} in \cref{l:ax_elprops},
  and \cref{l:ab_ba_char} \ref{i:bax_suba}
  implies that~\ref{i:x_subalg} and~\ref{i:x_bax}
  are equivalent.
\end{proof}

When the equivalent conditions in \cref{t:x_sct}
hold, then the pair $(\px, \Clpt(\px))$ forms
a supercharacter theory.
Indeed, Condition~\ref{i:x_equ}
is actually equivalent to the \hyperref[def:sct]{definition}
given in the introduction.

We leave the following dual as an exercise for the reader:

\begin{thm}\label{t:k_sct}
  Let $\pk$ be a $G$-invariant partition of $G$.
  Then the following are equivalent:
  \begin{enums}
  \item \label{i:k_equ} $\abs{\pk} = \abs{\Irpt(\pk)}$.
  \item \label{i:k_subalg}
         The linear span in $\compl G$
        of the sums $\hat{K}$ for $K\in \pk$
        is a unital subalgebra of $\Z(\compl G)$
  \item \label{i:k_abk} $\pk= \Clpt(\Irpt(\pk))$.
  \end{enums}
\end{thm}

Condition~\ref{i:k_subalg} means that the superclass sums
$\hat{K}$ span a so-called Schur ring over $G$.
(The connection between Schur rings and supercharacter theories 
 is explained with more details by Hendrickson~\cite{Hendrickson12}.)

Notice that when $\px= \Irpt(\Clpt(\px))$,
then $\Clpt(\px)=\Clpt(\Irpt(\Clpt(\px)))$.
Thus the conditions in \cref{t:x_sct}
imply that the conditions of \cref{t:k_sct}
hold for $\pk:= \Clpt(\px)$,
and conversely, the conditions in \cref{t:k_sct}
imply the conditions in \cref{t:x_sct}
for $\px:=\Irpt(\pk)$.

We have by now proved all the claims in 
\cref{t:main} from the introduction.
We have also done all the work for the following
characterizations of supercharacter theories:

\begin{thm}\label{t:superch}
Let $\px$ be a partition of\/ $\Irr G$ and 
$\pk$ a $G$-invariant partition of $G$.
Then the following are equivalent:
\begin{enums}
\item \label{i:scdef}
      $\abs{\pk}=\abs{\px}$ 
      and
      $\pk = \Clpt(\px)$.
\item \label{i:scomegas}
      $\abs{\pk}=\abs{\px}$
      and
      $\px = \Irpt(\pk)$.
\item \label{i:scsymequ}
      $\pk = \Clpt(\px)$ and $\px=\Irpt(\pk)$.
\item \label{i:classalg}
      The linear span in $\compl G$ of the 
      elements $\hat{K}$ for $K\in \pk $
      is a subalgebra of $\Z(\compl G)$,
      and $\px = \Irpt(\pk)$.
\item \label{i:charalg} 
      The linear span in the space of class functions
      of the characters $\sigma_X$ for $X\in \px$
      is a subalgebra of the  ring of all class functions
      and $\pk = \Clpt(\px)$.
\end{enums}
\end{thm}
\begin{proof}
  \ref{i:scdef}, \ref{i:scomegas}, \ref{i:scsymequ} and
  \ref{i:charalg} are equivalent
  by \cref{t:x_sct},
  and \cref{t:k_sct} yields the equivalence with
  \ref{i:classalg}.
\end{proof}

\begin{remark}
  The conditions of \cref{t:superch}
  are also equivalent to the following:
  \begin{enums}[label= \textup{(\alph*\ensuremath{'})}]
  \item \label{i:scdefw}
      $\abs{\pk}=\abs{\px}$ 
      and
      $\pk \preceq \Clpt(\px)$.      
  \item \label{i:scomegasw}
      $\abs{\pk}=\abs{\px}$
      and
      $\px \preceq \Irpt(\pk)$.            
  \item \label{i:scsymleq}
      $\pk \preceq \Clpt(\px)$
      and
      $\px \preceq \Irpt(\pk)$.      
  \end{enums}
\end{remark}
 \begin{proof} 
  Since we always have 
  $\abs{\px} \leq \abs{\Clpt(\px)}$,
  condition~\ref{i:scdef} is in fact equivalent
  to the weaker~\ref{i:scdefw}.
  Similarly, \ref{i:scomegas} and \ref{i:scomegasw}
  are equivalent.
  The equivalence of
  \ref{i:scsymleq} and \ref{i:scsymequ}
  follows from \cref{c:ab_ba_ref}
  (or otherwise).
\end{proof}

The last theorem emphasizes the duality between
superclasses and supercharacters.
The duality between classes and characters 
is maybe obscured by the following difference:
On the one side, we simply consider 
class sums $\hat{K} = \sum_{g\in K} g$,
while on the other side, 
we do not add simply the irreducible characters in a subset $X$,
but the multiples $\chi(1)\chi$.

When $\sct{X}{K}$ is a supercharacter theory,
and if $K$ and $L\in \pk$ are superclasses,
then $\hat{K}\hat{L}$ is a nonnegative integer 
linear combination of superclass sums, as is not difficult to see
\cite[Corollary~2.3]{DiaconisIsaacs08}.
But the product of two
supercharacters $\sigma_X$ and $\sigma_Y$ 
is in general only a rational combination 
of such supercharacters.
However, there is an easy remedy for this problem.
\begin{remark}
  Let $\sct{X}{K}$ be a supercharacter theory
  of $G$. For each $X\in \px$,
  define $d_X= \gcd\{ \chi(1)\mid \chi\in X \}$
  and $\tau_X = (1/d_X)\sigma_X$.
  Then every character that is constant on
  the members of $\pk$ is a nonnegative integer 
  linear combination of the $\tau_X$.
  In particular, this holds true for a product
  $\tau_X \tau_Y$.
\end{remark}
\begin{proof}
  Let $\gamma$ be a character that is constant on
  members of $\pk$.
  The $\sigma_X$, and thus the $\tau_X$, form a basis
  of the space of class functions constant on
  members of $\pk$.
  Thus there are $a_X\in \compl$ such that
  \[ \gamma = \sum_{X\in \px} a_X \tau_X.\]
  One the other hand, for $\chi\in X$ we have
  \[ (\gamma, \chi)= a_X(\tau_X,\chi)
                   =a_X \frac{\chi(1)}{d_X}
                   \in \nats,
  \]
  since $\gamma$ is a character.
  Choose $k_{\chi}\in \ints$ such that
  $d_X = \sum_{\chi\in X} k_{\chi} \chi(1)$.
  It follows
  \[ a_X = \frac{a_X}{d_X}\sum_{\chi\in X}k_{\chi}\chi(1)
        = \sum_{\chi\in X} k_{\chi} a_X \frac{\chi(1)}{d_X}\in \ints.\]
  Thus $a_X\in \nats$, as claimed.
\end{proof}
\begin{question}
  How are the supercharacters $\chi_{\lambda}$
  of algebra groups as defined 
  by Diaconis and Isaacs~\cite{DiaconisIsaacs08}
  related to the $\tau_X$?
  Is $\tau_X = \chi_{\lambda}$
  (where $X$ is the set of constituents
  of $\chi_{\lambda}$)?
\end{question}


\section{Computing all supercharacter theories of a finite group}
\label{sec:find_all_scts}

Let $(\px, \pk)$ and $(\pttf{Y}, \pttf{L})$ be
supercharacter theories of a group $G$.
We say that $(\px, \pk)$ is finer than
$(\pttf{Y}, \pttf{L})$ or that
$(\pttf{Y}, \pttf{L})$ is coarser than
$(\px, \pk)$, when
$\px$ is finer than $\pttf{Y}$.
It is clear from \cref{t:superch}
that  
$\px \preceq \pttf{Y}$ if and only if
$\pk \preceq \pttf{L}$ (for supercharacter theories).
(This has also been proved by Hendrickson~\cite{Hendrickson12}.)

It follows that the 
partially ordered set of supercharacter theories
embeds naturally 
into the partially ordered set of partitions
of $\Irr G$, or of $G$-invariant partitions of $G$.

The partitions of a given set form  a lattice. 
Let $(\px,\pk)$ and $(\pttf{L},\pttf{Y})$ be two supercharacter
theories.
Then
$(\px,\pk) \vee (\pttf{Y},\pttf{L})
 := (\px \vee \pttf{Y}, \pk \vee \pttf{L})$,
is also a supercharacter theory~\cite[Prop.~3.3]{Hendrickson12}.
It follows that the supercharacter theories of a group
form also a lattice.
Thus for every partition $\px$ of $\Irr G$,
there is a unique coarsest partition $\pttf{Y}$
refining $\px$, and such that $\pttf{Y}$
belongs to a supercharacter theory,
and similarly for $G$-invariant partitions of $G$.
The results of the last section provide a convenient way
to compute this supercharacter theory.

For example, start with an arbitrary 
$G$-invariant partition $\pk$ of $G$.
Apply in turns the maps $\Irpt$ and $\Clpt$ 
to form partitions of $\Irr G$ and of $G$.
Thus we get two chains of partitions as follows
(here and in the following,
we abbreviate $I:=\Irpt$ and $C:= \Clpt$):
\[ \begin{tikzcd}[column sep=0.01em]
      \pk \arrow[d]        & \succeq
       & CI(\pk)
         \arrow[d] & \succeq 
       & \dotsb & \succeq  
       & (C  I)^n(\pk)
          \arrow[d]
       & \succeq & \dotsb
      \\
       I(\pk) \arrow[urr]   & \succeq 
      & ICI(\pk)  \arrow[urr]  & \succeq 
      & \dotsb \arrow[urr] & \succeq
      & {I(C I)^n(\pk)}
         \arrow[urr]
      & \succeq & \dotsb .
\end{tikzcd}
\]
The number of blocks in a partition increases 
along the arrows.
The first two partitions 
that have the same numbers of blocks, 
form a supercharacter theory.

Of course, we can as well start with a partition
of $\Irr G$.

\begin{cor}\label{c:refine_to_sct}\hfill 
\begin{enums}
\item Let $\pk$ be a $G$-invariant partition of $G$.
      Choose $n$ such that 
      $(C I)^{n+1}(\pk) 
       = (CI)^n(\pk)$.
      Then $\big( I(CI)^n(\pk), (CI)^n(\pk) \big)$ is the coarsest
      supercharacter theory whose
      class partition refines $\pk$.
\item Let $\px$ be a partition of $\Irr G$.
      Choose $n$ such that $(IC)^{n+1}(\px) = (IC)^n(\px)$.
      Then $\big( (IC)^n(\px), C(IC)^n(\px) \big)$ is the coarsest
      supercharacter theory whose
      character partition refines $\px$.
\end{enums}
\end{cor}
\begin{proof} 
    This follows from the theorems in the last section.
\end{proof}

The number of steps to reach a supercharacter theory 
is in theory bounded by 
$k(G)-\abs{\pk} = \abs{\Irr G} - \abs{\pk}$.
Usually, one needs much less steps since the number
of sets in the involved partitions grows much faster than
by one in a step.
For example, consider the alternating group $G=A_7$ on $7$ elements,
a group with $9$ conjugacy classes.
There are $4140$ partitions $\pk$ of $G$ into $G$-invariant sets and
such that
$\{1\}$ is a block of the partition.
Of these, $3$ are supercharacter theories, 
and another $3807$ partitions are such that $\Irpt(\pk)$ belongs to a
supercharacter theory.
There are $292$ partitions where we need $2$ steps to reach a 
supercharacter theory, $31$ partitions where we need $3$ steps,
and only $7$ partitions where we need $4$ steps.
This is the worst case.
(Of course, we always need to compute one step more to see
that we have actually reached a supercharacter theory.)

As another example, let $G$ be an elementary abelian group of
order~$16$.
Consider only partitions $\pk$ of the form $G = \{1\} \dcup S \dcup T$.
Up to automorphisms of the group, 
there are only $22$ such partitions,
and of these, $5$ are already supercharacter theories.
For all these partitions $\pk$, 
already $CI(\pk)$ defines a supercharacter theory.

\begin{cor}\label{c:find_superclass}
  Let $K\subseteq G$ be a normal subset.
  Then there is a supercharacter theory 
  $(\px,\pk)$ with $K \in \pk$ 
  if and only if 
  $K\in (CI)^n(\{K\})$ for all $n>0$.
  If this is the case, then
  $(CI)^n(\{K\})$ for $n$ large belongs to the coarsest
  supercharacter theory with $K$ as superclass.
  (An analogous statement holds for subsets $X\subseteq \Irr G$.)
\end{cor}
\begin{proof}
  This follows from \cref{c:refine_to_sct}
  and \cref{l:ab_ba_char}.
\end{proof}

In the following, we use the notations 
$\noids{G}:= G \setminus \{1\}$
and $\noids{\Irr}G:= \Irr G \setminus \{1\}$.

\begin{algo}\label{algo:allscts}
  To compute the set $\mathcal{S}$ 
  of all supercharacter theories of a given group $G$
  with known character table,
  do the following steps:  
  \begin{enumerate}[label=\arabic*.]
  \item For every nonempty, $G$-invariant subset 
        $S \subset \noids{G}$, such that $S$ contains at most
        half of the nontrivial conjugacy classes of $G$,
        do the following:
        compute the partitions 
        $I(\{S\})$, $CI(\{S\})$, $\dotsc$,
        until either $S\notin (CI)^n(\{S\})$ for some $n$,
        or the number of blocks in these partitions stabilizes. 
        In the latter case, we have found the coarsest supercharacter 
        theory, $\pttf{L}(S)$, which has $S$ as 
        one of its superclasses.
        Add each such $\pttf{L}(S)$ to the set $\mathcal{S}$.
  \item Form all possible meets of two or more members
        of $\mathcal{S}$
        (in the lattice of supercharacter theories)
        and add them to $\mathcal{S}$.
  \item Add the coarse supercharacter theory 
       \[
          \mathsf{M}(G)
          := \big( \big\{\, \{1\},\, \noids{\Irr} G  
                   \,\big\}
                   , \;
                   \big\{\, \{1\},\, \noids{G}
                  \, \big\}
             \big)
       \]
       to $\mathcal{S}$.
  \end{enumerate}
\end{algo}

To compute the meet of some supercharacter theories
$(\px_i, \pk_i)$, 
we first compute the meet
$\pk:=\bigwedge_i \pk_i$ in the partition lattice,
and then use \cref{c:refine_to_sct}
to refine $\pk$ to a supercharacter theory.

Also in practice, one works only on the character table,
and represents $G$-invariant subsets $S\subset G$
by a list of class positions.

\begin{proof}[Proof of correctness of \cref{algo:allscts}]
  Let $\pk = \{S_1, \dotsc, S_r\}$ be a partition of the set $G$
  into disjoint, nonempty sets, and assume that $r\geq 2$.
  Then, in the partition lattice, $\pk$ equals the meet
  \[ \pk = \{S_1,\, G \setminus S_1\} \wedge \dotsb \wedge
           \{S_{r-1},\, G\setminus S_{r-1}\}.                           
  \]
  Notice that we have omitted exactly one set $S_r$ here.
  
  Now let $(\px, \pk)$ be a supercharacter theory, 
  and write $\pk= \{S_1,\allowbreak\, \dotsc,\allowbreak\, S_r\}$.
  Then $\{1\}\in \pk$, say $S_1=\{1\}$.
  Assume that $\pk \neq \mathsf{M}(G)$, that is, 
  $r=\abs{\pk}\geq 3$.
  Clearly, we have 
  $\pk \preceq \pttf{L}(S_i) \preceq \{S_i, G\setminus S_i\}$
  for each $S_i\in \pk$.
  It follows that
  \[  \pk = \pttf{L}(S_1) \wedge \dotsb \wedge
            \pttf{L}(S_{r-1}).
  \]
  (This is true both in the partition lattice and 
  in the lattice of supercharacter theories.)
  There is at most one $S_i\in \pk$ that contains
    more than half of the nontrivial conjugacy classes of $G$,
  and we may assume that this is $S_{r}$.
  Since every supercharacter theory refines 
  $\pttf{L}(\{1\}) = \mathsf{M}(G)$,
  we can omit $\pttf{L}(\{1\})$ from the above meet.
  It follows that $\pk$ is a meet of partitions
  $\pttf{L}(S_i)$ that were found in the first step.
  Thus the algorithm works.
\end{proof}

Let us compare this algorithm to the one 
proposed by Hendrickson~\cite[A.9]{Hendrickson08}.
Hendrickson first determines certain 
\emph{good} subsets of $G$,
by running through all nonempty,
$G$-invariant subsetes of $G$.
In a second step, Hendrickson determines 
all partitions into good subsets belonging to
supercharacter theories.

A dual version of Hendrickson's algorithm was proposed by 
Lamar~\cite{Lamar18}, building on ideas used by
Burkett, Lamar, Lewis and Wynn~\cite{BurkettLLW17}
to show that $G=\operatorname{Sp}(6,2)$ 
has only two supercharacter theories.
Here one uses \emph{good} subsets of $\Irr G$ 
and \emph{admissible} partial partitions of $\Irr G$.

It follows from \cref{l:ab_ba_char} that
a subset $K$ of $G$ is good in Hendrickson's sense
if and only if $K \in CI(\{K\})$,
and that a subset $X \subseteq \Irr G$ is good in the sense
of Burkett et al.~\cite{BurkettLLW17} if and only if 
$X \in IC(\{X\})$. 
The original definitions were different, and the resulting algorithms 
to decide whether some subset is good are less effective.
For example, to check whether a subset $X\subseteq \Irr G$ is good,
Burkett et al. check whether the characters $\chi(1)\chi$ and
$\psi(1)\psi $ occur with the same coefficients in the powers
$\sigma_X^k$ for all $\chi$, $\psi\in X$.

If $K\subseteq G$ is a superclass in some supercharacter theory,
then $K$ is necessarily good, but not conversely.
The first step of \cref{algo:allscts} decides directly if
a certain normal subset
$S\subseteq G$ can be a superclass, and if so, finds the coarsest 
supercharacter theory of $G$ that contains $S$ as a superclass.
In particular, the first step 
of our algorithm finds all possible superclasses
of supercharacter theories,
except the trivial superclass $\noids{G}$
and possibly some large superclasses
containing more than $(k(G)-1)/2$ conjugacy classes.
(In principle, it is even possible that such a superclass
is not contained in any supercharacter theory
determined in the first step, but then it has to be contained
in an even larger superclass.)

Of course, for some sets $S$, we will perhaps need more
steps until we reach a supercharacter theory or see that $S$ is not a 
superclass.
As mentioned before, usually very few steps suffice,
but I do not understand how to predict this.

A variant of the second step in \cref{algo:allscts}
would be to try all possible refinements of the
superclass partitions found in the first step,
using only possible superclasses.
However,
it is often the case that the second step yields no
or only few new supercharacter theories.

Of course, \cref{algo:allscts}
is still rather naive. 
For example, many sets $S$ will lead to the same supercharacter
theory $\pttf{L}(S)$.
One possible speedup is to use automorphisms
of the character table (see below).


\section{Table automorphisms}
\label{sec:tableauts}

A \defemph{table automorphism} of the character table
of some finite group $G$ is a pair
$(\sigma,\tau)$, where $\sigma$ is a permutation of
$\Irr G$, and $\tau$ a permutation of $\Cl(G)$,
the conjugacy classes of $G$, and such that
$\chi^{\sigma}(g^{\tau}) = \chi(g)$
for all $\chi\in \Irr G$ and $g\in G$
(where $g^{\tau}$ denotes an element in $(g^G)^{\tau}$).
Since the character table as a matrix is invertible,
it follows that $\sigma$ and $\tau$ determine
each other uniquely, when $(\sigma,\tau)$
is a table automorphism. 
The set of all table automorphisms of a given character table
forms a group, which acts on $\Irr G$ and on $\Cl(G)$.
(To prevent possible confusion, we remark here that in the
computer algebra system GAP~\cite{GAP}, 
the attribute \texttt{AutomorphismsOfTable}
gives only those column permutations of a character table 
that preserve power maps, in particular element orders.)

By Brauer's permutation lemma~\cite[Theorem~6.32]{isaCTdov},
the number of irreducible characters fixed
by $\sigma$ equals the number of classes fixed by
$\tau$ when $(\sigma,\tau)$ is a table automorphism.

The next lemma is the common generalization of two constructions
of supercharacter theories described in the introduction of
the Diaconis-Isaacs paper~\cite{DiaconisIsaacs08}.
We record it here for the sake of completeness
and reference below.
\begin{lemma}\label{l:sctfromtabauts} 
If $A$ is a subgroup of table automorphisms,
then the partitions of $\Irr G$ and $\Cl(G)$ into orbits under $A$
form a supercharacter theory.
\end{lemma} 
\begin{proof} 
Namely, every table automorphism fixes the class of 
the identity of $G$ (and the trivial character),
and thus every $\chi$ in an orbit $X$
has the same degree.
It is easy to see that the sum of the characters in an orbit
$X$ is constant along orbits of $A$ on $\Cl(G)$.
By Brauer's permutation lemma,
there are equal number of $A$-orbits on $\Irr G$ and
on $\Cl(G)$.
\end{proof}

\begin{lemma}
  Let $(\sigma,\tau)$ be a table automorphism of 
  the character table of $G$.
  \begin{enums}
  \item $\Clpt(\px^{\sigma}) = \Clpt(\px)^{\tau}$
        for any collection $\px$ of subsets of $\Irr G$.
  \item $\Irpt(\pk^{\tau})= \Irpt(\pk)^{\sigma}$
        for any collection $\pk$ of 
        $G$-invariant subsets of $G$.
  \item If $(\px,\pk)$ is a supercharacter theory of $G$,
        then so is $(\px^{\sigma},\pk^{\tau})$.
  \end{enums}
\end{lemma}
\begin{proof}
  The first and second item follows easily from the definitions.
  The last one is then clear in view of \cref{t:superch}.
\end{proof}

A special class of table automorphisms is induced 
by Galois automorphisms. 
Indeed, if $\sigma$ is a field automorphism of $\compl$,
then there is some integer $r$ coprime to $\abs{G}$
such that $(\sigma,\tau)$ is a table automorphism,
where $\tau$ is induced by the permutation
$g \mapsto g^r$.

\begin{lemma}
   Let $(\sigma,\tau)$ be a table automorphism
   of the character table of $G$,
   where $\sigma$ is induced by some field automorphism.
  \begin{enums}
  \item \label{i:gal_a} $\Clpt(\px^{\sigma}) = \Clpt(\px)$
        for any collection $\px$ of subsets of\/ $\Irr G$.
  \item \label{i:gal_b} $\Irpt(\pk^{\tau})= \Irpt(\pk)$
        for any collection $\pk$ of 
        $G$-invariant subsets of $G$.
  \item \cite[Theorem~2.2(e),(f)]{DiaconisIsaacs08}
        \label{i:gal_sct} 
        If $(\px,\pk)$ is a supercharacter theory of $G$,
        then $(\px^{\sigma},\pk^{\tau})= (\px,\pk)$.
  \end{enums}   
\end{lemma}
\begin{proof}
  Let $X\in \px$.
  We have 
  $\sigma_{X^{\sigma}} = \sigma\circ \sigma_X $
  as function $G\to\compl$.
  This shows~\ref{i:gal_a}.
  
  Similarly, the partition $\Irpt(\pk^{\tau})$
  is defined by maps
  $\alpha_{\widehat{K^{\tau}}}\colon \Irr G \to \compl$.
  We have
  \[ \alpha_{\widehat{K^{\tau}}}(\chi)
     = \frac{\chi(\widehat{K^{\tau}})}{\chi(1)}
     = \frac{1}{\chi(1)} \sum_{g\in K} \chi(g^r)
     = \frac{1}{\chi(1)} \sum_{g\in K} \chi(g)^{\sigma}
     = \alpha_{\widehat{K}}(\chi)^{\sigma},
  \]
  and so 
  $ \alpha_{\widehat{K^{\tau}}} 
    = \sigma\circ \alpha_{\widehat{K}} $.
  This yields~\ref{i:gal_b},
  and \ref{i:gal_sct} follows
  from~\ref{i:gal_a} and \ref{i:gal_b}.
\end{proof}

The last two results can be used to speed up
\cref{algo:allscts}.


\section{Discussion}
\label{sec:nrscts}

\ctable[ caption = {Number of supercharacter theories %
                    for some projective special linear groups
                    $G= \PSL(n,q)$%
                   },
         label = tab:psl2,
    ]{rrr@{\hspace{1em}}lrr}{
    }{   \FL
       &   &  &    &  \multicolumn{2}{c}{Supercharacter theories}  \NN
   $n$ & $q$ & $k(G)$ & Aut(CT($G$)) & \qquad all & from table auts \ML
  2 & 7 & 6 & $C_2$ & 4 & 2 \NN
    & 8 & 9 & $C_3\times C_3$ & 7 & 4 \NN 
    &  9 & 7 & $C_2\times C_2$ & 7 & 4 \NN
    & 11 & 8 & $C_2\times C_2$ & 13 & 4 \NN 
    & 13 & 9 & $C_6$ & 13 & 4 \NN 
    & 16 & 17 & $C_8\times C_4$ & 33 & 12 \NN 
    & 17 & 11 & $C_6\times C_2$ & 25 & 8 \NN 
    & 19 & 12 & $C_6\times C_2$ & 34 & 8 \NN 
    & 23 & 14 & $C_{10}\times C_2$ & 41 & 8 \NN 
    & 25 & 15 & $C_6\times C_2\times C_2$ & 81 & 16 \NN 
    & 27 & 16 & $C_6\times C_6$ & 45 & 16 \NN 
    & 29 & 17 & $C_{12}\times C_2$ & 89 & 12 \NN 
    & 31 & 18 & $C_{4}\times C_4\times  C_2$ & 161 & 18 \NN 
    & 37 & 21 & $C_{18}\times C_3$ & 76 & 12 \NN 
    & 41 & 23 & $C_{12}\times C_2 \times C_2$ & 307 & 24 \NN 
    & 43 & 24 & $C_{30}\times C_2$ & 100 & 16 \ML
  3 &  3 & 12 & $ C_4 \times C_2 $ & 7 & 6 \NN
    & 4  & 10 & $S_3 \times C_2 \times C_2 $ & 23 & 20 \NN 
    & 7  & 22 & $ S_3 \times C_{12} \times C_2 $ & 121 & 60 \LL
}

\ctable[ caption = Number of supercharacter theories %
                   for some alternating groups,
         label = tab:alt,
    ]{lrcrr}{
    }{   \FL
            &  &    &  \multicolumn{2}{c}{Supercharacter theories}  \NN
    $G$     & $k(G)$ & Aut(CT($G$)) & \qquad all & from table auts \ML
    $A_4 $ &  4 & $C_2$ & 3 & 2  \NN
    $A_5 $ &  5 & $C_2$ & 3 & 2  \NN 
    $A_6$  &  7 & $C_2\times C_2$ & 7 & 4 \NN
    $A_7$  &  9 & $C_2 $          & 3 & 2 \NN 
    $A_8$  & 14 & $C_2\times C_2$ & 5 & 4 \NN 
    $A_9$  & 18 & $C_2\times C_2$ & 5 & 4 \NN 
    $A_{10}$ & 24 & $C_2\times C_2$ & 5 & 4 \LL
}

\ctable[ caption = Number of supercharacter theories %
                   for some sporadic simple groups,
         label = tab:spor,
    ]{lrcrr}{\tnote{ $M_k$: Mathieu groups,  
                     $J_k$: Janko groups, 
                     $HS$: Higman-Sims group,
                     $McL$: McLaughlin group.}
    }{   \FL 
            &  &    &  \multicolumn{2}{c}{Supercharacter theories}  \NN
   $G$     & $k(G)$ & Aut(CT($G$)) & \qquad all & from table auts \ML
  $M_{11}$  & 10     & $ C_2 \times C_2 $ & 5  & 4  \NN
  $M_{12}$  & 15     & $ C_2 \times C_2 $ & 5  & 4  \NN
  $M_{22}$  & 12     & $ C_2 \times C_2 $ & 5  & 4  \NN
  $M_{23}$  & 17     & $ (C_2)^4 $        & 17 & 16 \NN
  $M_{24} $ & 26     & $ (C_2)^3 $        & 9  & 8  \ML
  $ J_1  $  & 15     & $ C_6$             & 5  & 4  \NN
  $J_2$     & 21     & $ C_2 $            & 3  & 2  \NN
  $J_3$     & 21     & $ C_6 \times C_2 \times C_2 $ & 17 & 16 \ML
  $HS$      & 24     & $ (C_2)^3 $        & 9  & 8  \NN
  $McL$     & 24     & $ (C_2)^4 $        & 17 & 16 \LL
}

    By using an implementation of \cref{algo:allscts}
    into the computer algebra system GAP,
    all supercharacter theories
    for some nonabelian simple groups were computed. 
    There is always the coarsest supercharacter theory~$\mathsf{M}(G)$,
    and there are the supercharacter theories
    that can be constructed from a subgroup~$A$
    of the table automorphism group, as described in
    \cref{l:sctfromtabauts}.
    (For $A=\{1\}$, this includes the 
    other trivial supercharacter theory.)
    For a number of simple groups, these are actually all 
    supercharacter theories. 
    This includes all the sporadic simple groups with $26$
    or fewer conjugacy classes
    (the five Mathieu groups, the Janko groups $J_1$, $J_2$,
    $J_3$, the Higman-Sims group $HS$ and the 
    McLaughlin group $McL$),
    the alternating groups $A_n$ with $n\leq 10$ and $n\neq 6$,
    the Tits group $T$ and the exceptional group $G_2(3)$,
    and the unitary groups $U_n(q)$ with
    $(n,q) = (3,3)$, $(3,4)$, $(3,5)$, $(4,2)$, $(4,3)$.  
    (In particular, $J_2$ and $U_4(2)$ have exactly $3$ supercharacter 
    theories.)
    On the other hand, for the projective special linear groups,
    there are usually more supercharacter theories.
    \cref{tab:psl2} contains the number of 
    supercharacter theories of $\PSL(n,q)$ for some values of
    $q$ and $n=2$, $3$ (second last column).
    The last column indicates how many supercharacter theories
    can be obtained from a subgroup~$A$ 
    of the table automorphism group $\Aut(\operatorname{CT}(G))$, 
    as in \cref{l:sctfromtabauts}.
    Another simple group with nontrivial supercharacter theories
    not coming from table automorphisms is the Suzuki group
    $\Sz(8)$ which has $11$ different supercharacter theories,
    of which only $8$ come from table automorphisms.

    These tables suggest a number of problems and conjectures:
    \begin{conjecture}
        For all $n\geq 7$, 
        every supercharacter theory $ \neq \mathsf{M}(G)$
        of the alternating group~$A_n$
        can be constructed from a group 
        of table automorphisms.
    \end{conjecture}
    
    In view of the specifics of the representation theory
    of $A_n$ and $S_n$, this is equivalent to 
    a conjecture by J. Lamar~\cite[Conjecture~3.35]{Lamar18}.
    
    Closely related is the following conjecture:
    \begin{conjecture}
       For every $n\geq 7$, the symmetric group~$S_n$
       has exactly $4$ supercharacter theories.       
    \end{conjecture}
    
    (The normal subgroup $A_n$ of $S_n$ yields
     $2$ more supercharacter theories besides the trivial ones
     via known constructions~\cite{Hendrickson12}.)
     
    \begin{problem}
       Give a generic construction of some supercharacter theory
       of $\PSL(n,q)$ (or $\Sz(q)$) other than supercharacter theories
       from orbits of table automorphisms and the trivial
       supercharacter theory.       
    \end{problem}
    
    The last problem can of course be asked for any other class of
    groups.
    In fact, the original motivation of the theory was to construct
    a supercharacter theory for the unitriangular groups
    $\UT(n,q)$, where the full character table is not available,
    and Diaconis and Isaacs constructed more generally a supercharacter
    theory of \emph{algebra groups}~\cite{DiaconisIsaacs08}.
    Since then, similar constructions have been given for many
    other unipotent
    groups~\cite{AndreFreitasNeto15,Andrews15,Panov18,Thiem18}.
     
    When $G$ is not simple, then there are several constructions
    of supercharacter theories
    using normal subgroups~\cite{Aliniaeifard17,Hendrickson12},
    and table automorphisms yield also supercharacter theories.
    Moreover, one may form meets and joins of these
    supercharacter theories.
    So in general, it may be difficult to decide whether
    some proposed construction of supercharacter theories
    yields something that can not be obtained from the 
    other available methods.

\printbibliography

\end{document}

%% file: layout_scrartcl.tex
\usepackage{lmodern}
\usepackage{microtype}

\makeatletter
\renewcommand*\author[1]{%
  \stepcounter{author}%
  \ifnum\c@author=1
    \gdef\@author{#1}%
  \else
    \xdef\@author{\unexpanded\expandafter{\@author\and#1}}%
  \fi
  \csgdef{author@\the\c@author}{#1}}
\newcommand*\email[1]{%
  \csgdef{email@\the\c@author}{#1}}
\newcommand*\address[1]{%
  \csgdef{address@\the\c@author}{#1}}
\AtEndDocument{%
  \xdef\author@count{\the\c@author}%
  \c@author=1
  \print@authors}
\newcommand*\print@authors{%
  \ifnum\c@author>\author@count
  \else
    \print@author{\the\c@author}%
    \advance\c@author by 1
    \expandafter\print@authors
  \fi}
\newcommand*\print@author[1]{%
  \par\medskip
  \noindent
  \begin{tabular}{@{}l@{}}%
    \csuse{author@#1}\\
    \csuse{address@#1}\\
    \textit{e-mail:}
    \texttt{\csuse{email@#1}}
  \end{tabular}
}

\renewcommand*{\title}[2][]{%
  \ifstrempty{#1}{
  \gdef\shorttitle{#2}\gdef\@title{#2}
  }{\gdef\shorttitle{#1}\gdef\@title{#2}%
  }
}
\makeatother

\RequirePackage{xpatch}

\RequirePackage{scrlayer-scrpage}

\ohead{\pagemark}
\makeatletter 
\cehead[]{\let\thanks\@gobble\@author}
\cohead[]{\shorttitle}
\makeatother
\ofoot[]{}

\makeatletter
\xpatchcmd{\swappedhead}%
         {~}%
         {.~}
         {}%
         {\typeout{failed to patch swappedhead}}
\makeatother

\KOMAoptions{twoside=semi,headings=small, numbers=enddot}

\setkomafont{disposition}{\normalcolor\bfseries}
\addtokomafont{title}{} 

\setkomafont{author}{\large}
\setkomafont{date}{\normalfont}

\newkomafont{abstract}{\normalfont\footnotesize}
\newkomafont{abstracttitle}{%
   \usekomafont{disposition}\usesizeofkomafont{abstract}%
}

\renewenvironment{abstract}%
  {\usekomafont{abstract}
   \begin{description}[%
        style=unboxed,
        leftmargin=3em,
        labelindent=\leftmargin,
        rightmargin=\leftmargin,
        labelwidth=0pt,        
        parsep=0pt,
        listparindent=\parindent, 
        font=\normalfont\usekomafont{abstracttitle}]
   \item[\abstractname.]}%
  {\end{description}}

\newcommand*{\keywordsname}{Keywords}
\newcommand*{\mscname}[1]{Mathematics Subject Classification (#1)}
  
\newcommand{\keywords}[1]{\item[\keywordsname.] #1.}  
\newcommand{\subjclass}[2][2010]{%
\item[\mscname{#1}.] #2.}